\documentclass[reqno,twoside]{amsart}

\usepackage{amsfonts}
\usepackage{amsmath}
\usepackage{amssymb}
\usepackage{amsthm}
\usepackage{mathtools}
\usepackage{dsfont}
\usepackage{etoolbox}
\usepackage{color}
\usepackage{mathrsfs}

\usepackage[dvips,bottom=1.4in,right=1.1in,top=1in, left=1.1in]{geometry}

\makeatletter
\def\eqnarray{\stepcounter{equation}\let\@currentlabel=\theequation
\global\@eqnswtrue
\tabskip\@centering\let\\=\@eqncr
$$\halign to \displaywidth\bgroup\hfil\global\@eqcnt\z@
  $\displaystyle\tabskip\z@{##}$&\global\@eqcnt\@ne
  \hfil$\displaystyle{{}##{}}$\hfil
  &\global\@eqcnt\tw@ $\displaystyle{##}$\hfil
  \tabskip\@centering&\llap{##}\tabskip\z@\cr}

\def\endeqnarray{\@@eqncr\egroup
      \global\advance\c@equation\m@ne$$\global\@ignoretrue}

\newtheorem{theorem}{Theorem}[section]

\newtheorem{definition}[theorem]{Definition}

\newtheorem{remark}[theorem]{Remark}
\numberwithin{equation}{section}

\def\RR{{\mathbb{R}}}

\def\Om{\Omega}

\def\bOm{\overline{\Om}}

\newcommand{\norm}[2]{{\left\|#1\right\|}_{#2}}
\newcommand{\fl}[2]{(-\Delta)^#1 #2}

\title[Local regularity for fractional heat equations]{Local regularity for fractional heat equations}

\author{Umberto Biccari}

\address{Umberto Biccari, DeustoTech, University of Deusto, 48007 Bilbao, Basque Country, Spain.}
\address{Umberto Biccari, Facultad de Ingenier\'{\i}a, Universidad de Deusto, Avda Universidades 24, 48007 Bilbao, Basque Country, Spain.}
\email{umberto.biccari@deusto.es, u.biccari@gmail.com}

\author{Mahamadi Warma}

\address{Mahamadi Warma, University of Puerto Rico  (Rio Piedras Campus), College of Natural Sciences,
Department of Mathematics, PO Box 70377 San Juan PR
00936-8377 (USA). }
\email{mahamadi.warma1@upr.edu, mjwarma@gmail.com}

\author{Enrique Zuazua}
\address{Enrique Zuazua, DeustoTech, University of Deusto, 48007 Bilbao, Basque Country, Spain.}
\address{Enrique Zuazua, Facultad de Ingenier\'{\i}a, Universidad de Deusto, Avda Universidades 24, 48007 Bilbao, Basque Country, Spain.}
\address{Enrique Zuazua, Departamento de Matem\'aticas, Universidad Aut\'onoma de Madrid, Campus de Cantoblanco,
28049, Madrid, Spain}
\email{enrique.zuazua@deusto.es, enrique.zuazua@uam.es}

\thanks{The work of Umberto Biccari was partially supported by the Advanced Grant DYCON (Dynamic Control) of the European Research Council Executive Agency, by the MTM2014-52347 Grant of the MINECO (Spain) and by the Air Force Office of Scientific Research under the Award No: FA9550-15-1-0027. The work of Mahamadi Warma was partially supported by the Air Force Office of Scientific Research under the Award No: FA9550-15-1-0027. The work of Enrique Zuazua was partially supported by the Advanced Grant DYCON (Dynamic Control) of the European Research Council Executive Agency, FA9550-15-1-0027 of AFOSR, FA9550-14-1-0214 of the EOARD-AFOSR, the MTM2014-52347 Grant of the MINECO (Spain) and ICON of the French ANR}



\begin{document}

\begin{abstract}
We prove the maximal local regularity of weak solutions to the parabolic problem associated with the fractional Laplacian with homogeneous Dirichlet boundary conditions on an arbitrary bounded open set $\Omega\subset\RR^N$.  Proofs combine classical abstract regularity results for parabolic equations with some new local regularity results for the associated elliptic problems. 
\end{abstract}

\maketitle

\begin{center}
{\it Dedicated to Enrique Fern\'andez-Cara on his 60th birthday.}
\end{center}

\section{Introduction}\label{intro}

The aim of the present paper is to study the local regularity of weak solutions to the following parabolic problem
\begin{align}\label{DP}
	\begin{cases}
		u_t+\fl{s}{u}=f &\mbox{ in }\;\Omega\times(0,T)=:\Omega_T,
		\\
		u\equiv 0 &\mbox{ on }\;(\RR^N\setminus\Omega)\times(0,T),
		\\
		u(\cdot,0)\equiv 0 &\mbox{ in }\;\Omega,
	\end{cases}
\end{align}
where $\Omega\subset\RR^N$ is an arbitrary bounded open set, $f$ is a given distribution and, for all $s\in(0,1)$, $(-\Delta)^s$ denotes the fractional Laplace operator, which is defined as the following singular integral 
\begin{align}\label{fl}
(-\Delta)^su(x):=C_{N,s}\,\mbox{P.V.}\int_{\RR^N}\frac{u(x)-u(y)}{|x-y|^{N+2s}}\;dy,\;\;x\in\RR^N.
\end{align}
In \eqref{fl}, $C_{N,s}$ is a normalization constant given by
\begin{align*}
C_{N,s}:=\frac{s2^{2s}\Gamma\left(\frac{2s+N}{2}\right)}{\pi^{\frac
N2}\Gamma(1-s)},
\end{align*}
$\Gamma $ being the usual Gamma function. 

We are interested in analyzing the local regularity for solutions to the parabolic problem \eqref{DP}. 

We first introduce the functional setting. Given $\Omega\subset\RR^N$, an arbitrary open set, for $p\in (1,\infty )$ and $s\in (0,1)$, we denote by
\begin{equation*}
W^{s,p}(\Omega ):=\left\{ u\in L^{p}(\Omega):\;\int_{\Omega}\int_{\Omega }
\frac{|u(x)-u(y)|^{p}}{|x-y|^{N+ps}}dxdy<\infty \right\},
\end{equation*}
the fractional order Sobolev space endowed with the norm
\begin{equation*}
\Vert u\Vert _{W^{s,p}(\Omega )}:=\left( \int_{\Omega }|u|^{p}\;dx+\int_{\Omega }\int_{\Omega }\frac{|u(x)-u(y)|^{p}}{|x-y|^{N+ps}}
dxdy\right) ^{\frac{1}{p}}.
\end{equation*}

\noindent We let
\begin{align*}
W_0^{s,p}(\bOm):=\Big\{u\in W^{s,p}(\RR^N):\; u=0\;\mbox{ on }\;\RR^N\setminus\Omega\Big\},
\end{align*}
and we shall denote by $W^{-s,2}(\bOm)$ the dual of the Hilbert space $W_0^{s,2}(\bOm)$, that is, $W^{-s,2}(\bOm):=(W_0^{s,2}(\bOm))^\star$. The following continuous embeddings hold
\begin{align*}
W_0^{s,2}(\bOm)\hookrightarrow L^2(\Omega)\hookrightarrow W^{-s,2}(\bOm).
\end{align*}

\noindent Next, if $s>1$ is not an integer, we write $s=m+\sigma$ where $m$ is an integer and $0<\sigma<1$. In this case
\begin{align*}
W^{s,p}(\Omega):=\Big\{u\in W^{m,p}(\Omega):\; D^\alpha u\in W^{\sigma,p}(\Omega)\;\mbox{ for any }\;\alpha\;\mbox{ such that }\;|\alpha|=m\Big\}.
\end{align*}
Then $W^{s,p}(\Omega)$ is a Banach space with respect to the norm
\begin{align*}
\|u\|_{W^{s,p}(\Omega)}:=\left(\|u\|_{W^{m,p}(\Omega)}^p+\sum_{|\alpha|=m}\|D^\alpha u\|_{W^{\sigma,p}(\Omega)}^p\right)^{\frac 1p}.
\end{align*}

\noindent If $s=m$ is an integer, then $W^{s,p}(\Omega)$ coincides with the classical integral order Sobolev space $W^{m,p}(\Omega)$. 

We also recall the following definition of the Besov space $B^{s}_{p,q}$, according to \cite[Chapter V, Section 5.1, Formula (60)]{STEIN}:
\begin{align}\label{besov-def}
	B^{s}_{p,q}(\RR^N) :=\left\{u\in L^p(\RR^N):\; \left(\int_{\RR^N}\frac{\norm{u(x+y)-u(y)}{L^p(\RR^N)}^q}{|y|^{N+qs}}\,dy\right)^{\frac{1}{q}}<\infty \right\}, \;\;\;1\le p,q\le\infty,\;\; 0<s<1.
\end{align}  

Notice that, when $p=q$, we have $B^{s}_{p,p}(\RR^N) = W^{s,p}(\RR^N)$. Finally, we recall the definition of the following potential space 
\begin{align}\label{sp-stein}
	\mathscr{L}^p_{2s}(\RR^N):=\Big\{u\in L^p(\RR^N):\;  \fl{s}{u}\in L^p(\RR^N)\Big\},\;\;\;1\le p\le\infty, \;\;s\ge 0,
\end{align}
introduced, for example, in \cite[Chapter V, Section 3.3, Formula (38)]{STEIN}. Note that this same space is sometimes denoted as $H^s_p(\RR^N)$ (see, e.g., \cite[Section 1.3.2]{TRIEB}). 
Here we  adopt the notation $\mathscr{L}^p_{2s}(\RR^N)$.

Let us now introduce the notion of solution that we shall consider. Following \cite{LPPS}, we first consider weak solutions of \eqref{DP} defined as follows.

\begin{definition}\label{weak_sol_def_fe}
Let $f\in L^2((0,T);W^{-s,2}(\bOm))$. 
We say that $u\in L^2((0,T);W_0^{s,2}(\bOm))\cap C([0,T];L^2(\Omega))$ with $u_t\in L^2((0,T);W^{-s,2}(\bOm))$ is a finite energy solution to the parabolic problem \eqref{DP}, if the identity
\begin{align}\label{weak-sol-fe}
	&\int_0^T\int_{\Omega} u_tw\,dxdt + \frac{C_{N,s}}{2}\int_0^T\int_{\RR^N}\int_{\RR^N}\frac{(u(x)-u(y))(w(x)-w(y))}{|x-y|^{N+2s}}\,dxdydt \notag\\
	=& \int_0^T \langle f,v\rangle_{W^{-s,2}(\bOm),W_0^{s,2}(\bOm)}\,dt,
\end{align}
holds, for any $w\in L^2((0,T);W_0^{s,2}(\bOm))$, where $\langle \cdot,\cdot\rangle_{W^{-s,2}(\bOm),W_0^{s,2}(\bOm)}$ denotes the duality pairing between $W^{-s,2}(\bOm)$ and $W_0^{s,2}(\bOm)$.
\end{definition}

\begin{remark}
{\em We observe the following facts.
\begin{enumerate}
\item According to \cite[Theorem 10]{LPPS}, if $u\in L^2((0,T);W_0^{s,2}(\bOm))$ and $u_t\in L^2((0,T);W^{-s,2}(\bOm))$, then $u\in C([0,T];L^2(\Omega))$. Thus  the identity $u(\cdot,0)=0$ makes sense in $L^2(\Omega)$. 
\item When considering right hand side terms  $f\in L^p((0, T);L^p(\Omega))=L^p(\Omega\times(0,T))$ with $p\ge 2$, since we have the continuous embedding $L^p(\Omega\times(0,T))\hookrightarrow L^2((0,T);W^{-s,2}(\bOm))$, this notion of weak finite energy solution suffices.
\item When $f\in L^p(\Omega\times(0,T))$ with $1\le p< 2$, the regularity of the right hand side term does not suffice to define weak finite energy solutions as above. We shall rather consider those defined by duality or transposition.
\end{enumerate}
}
\end{remark}

Duality or transposition solutions of  \eqref{DP} are given by duality with respect to the following class of test functions
\begin{align*}
	\mathcal{P}(\Omega_T) = \Big\{\phi(\cdot,t)\in C^1((0,T),C_0^{\beta}(\Omega)) : \phi\textrm{ is a solution to Problem (P)}\Big\},
\end{align*}
where
\begin{align*}
	(P) = \begin{cases}
		-\phi_t+\fl{s}{\phi}=\psi &\mbox{ in }\;\Omega\times(0,T)=:\Omega_T,
		\\
		\phi\equiv 0 &\mbox{ on }\;(\RR^N\setminus\Omega)\times(0,T),
		\\
		\phi(\cdot,T)\equiv 0 &\mbox{ in }\;\Omega
	\end{cases}
\end{align*}
for $\psi\in C_0^{\infty}(\Omega_T)$.

\begin{definition}\label{weak_sol_def}
Let $f\in L^1(\Omega\times(0,T))$. We say that $u\in C([0,T];L^1(\Omega))$ is a weak duality or transposition solution to the parabolic problem \eqref{DP}, if the identity
\begin{align}\label{weak-sol}
	&\int_0^T\int_{\Omega} u\psi\,dxdt = \int_0^T\int_{\Omega} f\phi\,dxdt 
\end{align}
holds, for any $\phi\in\mathcal{P}(\Omega_T)$ and $\psi\in C_0^{\infty}(\Omega_T)$.
\end{definition}

\begin{remark}
{\em 
The existence and uniqueness of finite energy  weak solutions or the duality/transposition ones (depending on the regularity imposed on the right hand side term $f$) to problem \eqref{DP} is guaranteed by \cite[Theorem 26]{LPPS} and \cite[Theorem 28]{LPPS}, respectively.
If $f\in L^p(\Omega\times(0,T))$, with  $p\ge 2$, finite energy solutions of  \eqref{DP} will be considered while, if $1<p<2$, solutions will be understood in the sense of duality/transposition. In both cases we shall refer to them as weak solutions.
}
\end{remark}

\noindent Our first regularity result concerns the case $p=2$. It reads as follows:

\begin{theorem}\label{reg-thm-2}
Assume $f\in L^2(\Omega\times(0,T))$ and let $u\in L^2((0,T);W^{s,2}_0(\bOm))\cap C([0,T];L^2(\Omega))$ with $u_t\in L^2((0,T);W^{-s,2}(\bOm))$ be the unique finite energy solution of system \eqref{DP}. Then
\begin{align*}
	u\in L^2((0,T);W^{2s,2}_{\rm loc}(\Omega))\cap L^{\infty}((0,T);W^{s,2}_0(\bOm))\;\;\mbox{ and }\; u_t\in L^2(\Omega \times (0, T)).
\end{align*}
\end{theorem}

\noindent Theorem \ref{reg-thm-2} can be extended to the $L^p$-setting as follows.

\begin{theorem}\label{reg-thm-p}
Let $1<p<\infty$ and $f\in L^p(\Omega \times (0,T))$. Then, problem \eqref{DP} has a unique weak solution $u\in C([0,T];L^p(\Omega))$ such that $u\in L^p\Big((0,T);\mathscr{L}^p_{2s, {\rm loc}}(\Omega)\Big)$ and $u_t\in L^p(\Omega \times (0,T))$. As a consequence we have the following result.
\begin{enumerate}
\item If $1<p<2$ and $s\ne 1/2$, then $u\in L^p\Big((0,T);B^{2s}_{p,2, {\rm loc}}(\Omega)\Big)$.
	
\item If $1<p<2$ and $s= 1/2$, then $u\in L^p\Big((0,T);W^{2s,p}_{\rm loc}(\Omega)\Big)=L^p\Big((0,T);W^{1,p}_{\rm loc}(\Omega)\Big)$.
	
\item If $2\leq p<\infty$, then $u\in L^p\Big((0,T);W^{2s,p}_{\rm loc}(\Omega)\Big)$.
\end{enumerate}

\end{theorem}
In Theorem \ref{reg-thm-p}, with $\mathscr{L}^p_{2s, \textrm{loc}}(\Omega)$ we indicate the potential space 
\begin{align}\label{sp-stein-loc}
	\mathscr{L}^p_{2s, \textrm{loc}}(\Omega):=\Big\{ u\in L^p(\Omega): u\eta \in \mathscr{L}_{2s}^p(\RR^N) \;\textrm{ for any test function $\eta\in\mathcal{D}(\Omega)$} \Big\}.
\end{align}
Analogously, with $B^{2s}_{p,2, {\rm loc}}(\Omega)$ we indicate the Besov space
\begin{align}\label{besov-loc}
	B^{2s}_{p,2, {\rm loc}}(\Omega):=\Big\{ u\in L^p(\Omega): u\eta \in B^{2s}_{p,2}(\RR^N) \;\textrm{ for any test function $\eta\in\mathcal{D}(\Omega)$} \Big\}.
\end{align}

Moreover, our results guarantee that when the right hand side belongs to $L^p(\Omega \times (0, T))$ for $2\le p<\infty$ and for $1<p<2$, $s=1/2$, then the corresponding solution gains locally the maximum possible regularity, that is, it gains one time derivative  and up to $2s$ space derivatives, locally, in $L^p(\Om)$. For $1<p<2$ and $s\neq 1/2$, instead, the local regularity is obtained in the Besov space $B^{2s}_{p,2, {\rm loc}}(\Omega)$, which is strictly larger than $W^{2s,p}_{\textrm{loc}}(\Omega)$.

For the classical Laplace operator (which corresponds to the case $s=1$), this kind of results are standard, see  e.g., \cite[Theorem X.12]{brezis}, \cite[Section 9]{Gris-abs}, \cite[Section 4.1]{LSU}. Also, we recall \cite[Theorem 1]{lamberton} for a more general result in an abstract setting.

Theorems \ref{reg-thm-2} and \ref{reg-thm-p} are natural extensions of analogous results of local regularity for the elliptic problem associated to the fractional Laplacian on a bounded domain, which have been obtained recently in \cite{fl_reg,fl_reg_add}. 

In the recent years, research on regularity of heat equations involving non-local terms has been very active. For instance,  H\"older regularity  was proved  in \cite{FR,KS}. Boundary regularity has also been analyzed showing that, if $f=0$ and taking  initial data $u(\cdot,0)=u_0\in L^2(\Omega)$, the corresponding solution to \eqref{DP} is such that $u(\cdot,t)$ belongs to $C^s(\RR^N)$ for all $t>0$ and satisfies $u(\cdot,t)/\rho^s\in C^{s-\varepsilon}(\Omega)$ for any $\varepsilon>0$, $\rho(x)=\textrm{dist}(x,\partial\Omega)$ being the distance to the boundary function. Concerning regularity in the Sobolev setting, we refer instead to \cite[Theorem 26]{LPPS}, where it has been  proved the existence of a finite energy solution to \eqref{DP}, according to Definition \ref{weak_sol_def} above. However, to the best of our knowledge, our Theorems \ref{reg-thm-2} and \ref{reg-thm-p} providing maximal space-time local regularity are new.

The controllability of parabolic equations involving non-local terms has also been investigated. We refer for instance to \cite{FLZ} where null controllability issues were addressed for  heat equations involving  non-local lower order terms. On the other hand,  \cite{MZ,miller} dealt with the control of heat equations involving the \textit{spectral} fractional Laplacian (see \cite[Section 1]{MZ} for the definition of this operator), proving that null controllability holds for $s>1/2$, while for $s\leq 1/2$ the equation fails to be controllable. Notice that this operator does not coincide with \eqref{fl}.

The present paper is organized as follows. In Section \ref{elliptic_sec}, we will recall the sharp local regularity results obtained in \cite{fl_reg,fl_reg_add} for the elliptic problems associated to the fractional Laplacian. These results will be necessary in the proof of Theorems \ref{reg-thm-2} and \ref{reg-thm-p}. In Section \ref{L2-reg}, we give the proof of Theorem  \ref{reg-thm-2}, using the corresponding result for the classical Laplace operator in \cite[Section 7.1.3, Theorem 5]{evans}, employing a cut-off argument and using \cite[Theorem 1.2]{fl_reg}. In Section \ref{Lp-reg} we give the proof of Theorem \ref{reg-thm-p} by applying the results contained in \cite{lamberton}. Finally, in Section \ref{open_pb}, we present some open problems and perspectives that are closely related to our work.

\section{Regularity results for the elliptic problem}\label{elliptic_sec}

In this section, we recall some regularity results for weak solutions to the elliptic problem associated to the fractional Laplacian on a bounded open set. These results have been recently obtained in \cite{fl_reg,fl_reg_add}, and they will be fundamental in the proof of Theorems \ref{reg-thm-2} and \ref{reg-thm-p}. Therefore,  throughout this section we are going to consider the following elliptic problem

\begin{align}\label{DP-ell}
	\begin{cases}
		\fl{s}{u}=f &\mbox{ in }\;\Omega,
		\\
		u\equiv 0 &\mbox{ on }\;\RR^N\setminus\Omega.
	\end{cases}
\end{align}

\noindent Let us start by recalling the definition of a weak solution, according to \cite{fl_reg,LPPS}.

\begin{definition}\label{weak_sol_def-en}
Let $f\in W^{-s,2}(\bOm)$. A function $u\in W_0^{s,2}(\bOm)$ is said to be a finite energy solution to the Dirichlet problem \eqref{DP-ell} if for every $v\in W_0^{s,2}(\bOm)$, the equality
\begin{align}\label{wek-sol-en}
\frac{C_{N,s}}{2}\int_{\RR^N}\int_{\RR^N}\frac{(u(x)-u(y))(v(x)-v(y))}{|x-y|^{N+2s}}\;dxdy=\langle f,v\rangle_{W^{-s,2}(\bOm),W_0^{s,2}(\bOm)}
\end{align}
holds.
\end{definition}

We notice that, when $f\in L^p(\Omega)$ with $1<p<2$ and it does not belong to $W^{-s,2}(\bOm)$, it is not natural to consider finite energy solutions for the problem \eqref{DP-ell}. As for the parabolic problem above,  we shall introduce an alternative notion of solution. This will be given by duality with respect to the following class of test functions:
\begin{align*}
	\mathcal{T}(\Omega) = \Big\{\phi : \fl{s}{\phi} = \psi\;\;\textrm{ in }\;\Omega,\;\phi=0\;\;\textrm{ in }\;\RR^N\setminus\Omega,\;\psi\in C_0^{\infty}(\Omega)\Big\}.
\end{align*}

\begin{definition}\label{weak-sol-def}
Let $f\in L^1(\Omega)$. We say that $u\in L^1(\Omega)$ is a weak duality or transposition solution to \eqref{DP-ell} if the equality
\begin{align*}
	\int_{\Omega} u\psi\,dx = \int_{\Omega}f\phi\,dx,
\end{align*} 
holds for any $\phi\in\mathcal{T}(\Omega)$ and $\psi\in C_0^{\infty}(\Omega)$.
\end{definition}

The existence and uniqueness of finite energy  weak solutions or the duality/transposition ones (depending on the regularity imposed on the right hand side term $f$) to problem \eqref{DP-ell} are guaranteed by \cite[Theorem 12]{LPPS} and \cite[Theorem 23]{LPPS}, respectively.
If $f\in L^p(\Omega)$, with  $p\ge 2$, finite energy solutions of  \eqref{DP-ell} will be considered while, if $1<p<2$, solutions will be understood in the sense of duality/transposition. In both cases, we shall refer to them as weak solutions. Moreover, we notice that, according to Definition \ref{weak-sol-def}, duality solutions do not require that $f$ belongs to the dual space $W^{-s,2}(\bOm)$. Finally, we also notice that, if $f\in L^p(\Omega)$ with $p\geq 2$, we have the continuous embedding  $L^p(\Omega)\hookrightarrow L^2(\Omega)\hookrightarrow W^{-s,2}(\bOm)$, meaning that the property $f\in W^{-s,2}(\bOm)$ is automatically guaranteed.

 Concerning  the regularity of the solutions to \eqref{DP-ell}, the following result has been proved in \cite{fl_reg,fl_reg_add}.

\begin{theorem}[\bf $L^p$-Local elliptic regularity]\label{reg-p-ell}
Let $1<p<\infty$. Given $f\in L^p(\Omega)$, let $u$ be the unique weak solution to the Dirichlet problem \eqref{DP-ell}. Then $u\in \mathscr{L}^p_{2s, {\rm loc}}(\Omega)$. As a consequence we have the following result.
\begin{enumerate}
\item If $1<p<2$ and $s\ne 1/2$, then $u\in B^{2s}_{p,2, {\rm loc}}(\Omega)$.

\item If $1<p<2$ and $s= 1/2$, then $u\in W_{\rm loc}^{2s,p}(\Omega)=W_{\rm loc}^{1,p}(\Omega)$.	
	
\item If $2\leq p<\infty$, then $u\in W^{2s,p}_{\rm loc}(\Omega)$.
\end{enumerate}
\end{theorem}

The proof of Theorem \ref{reg-p-ell} requires a cut-off argument that allows us to reduce the problem to the whole space case, for which the result is already known. In particular, we have the following.

\begin{theorem}\label{re-r-N}
Let $1<p<\infty$. Given $F\in L^p(\RR^N)$, let $u$ be the unique weak solution to the fractional Poisson type equation
\begin{equation}\label{PE}
	\fl{s}{u}=F\;\;\mbox{ in }\;\RR^N.
\end{equation}
Then $u\in \mathscr{L}^p_{2s}(\RR^N)$. As a consequence we have the following. 
\begin{enumerate}
\item If $1<p<2$ and $s\ne 1/2$, then $u\in B^{2s}_{p,2}(\RR^N)$.
\item If $1<p<2$ and $s= 1/2$, then $u\in W^{2s,p}(\RR^N)=W^{1,p}(\RR^N)$.
\item If $2\leq p<\infty$, then $u\in W^{2s,p}(\RR^N)$.
\end{enumerate}
\end{theorem}

Theorem \ref{re-r-N} is a classical result whose proof can be done by combining several results on singular integrals and Fourier transform contained in \cite[Chapter V]{STEIN}. See also \cite{fl_reg,fl_reg_add}. In particular:

\begin{itemize}
\item If $1<p<2$ and $s\ne 1/2$, then the result follows from \cite[Chapter V, Section 5.3, Theorem 5(B)]{STEIN}, which provides the inclusion $\mathscr{L}^p_{2s}(\RR^N)\subset B^{2s}_{p,2}(\RR^N)$. Moreover, an explicit counterexample showing that sharper inclusions are not possible has been given in  \cite[Chapter V, Section 6.8]{STEIN}. 

\item If $1<p<2$ and $s= 1/2$, then applying \cite[Chapter V, Section 3.3, Theorem 3]{STEIN} we have $\mathscr{L}^p_{2s}(\RR^N)=\mathscr{L}^p_{1}(\RR^N)=W^{1,p}(\RR^N)$.

\item If $2\leq p<\infty$, then \cite[Chapter V, Section 5.3, Theorem 5(A)]{STEIN} yields $u\in B^{2s}_{p,p}(\RR^N)$ and this latter space, by definition, coincides with $W^{2s,p}(\RR^N)$  (see, e.g., \cite[Chapter V, Section 5.1, Formula (60)]{STEIN}).
\end{itemize}

While developing the cut-off argument that we mentioned above, as an intermediate step we need to show that $u\in W^{s,p}(\Omega)$. Notice that, for $p\geq 2$,  this is true for all weak solutions to \eqref{DP} by classical embedding results. When $1<p<2$, instead, according to \cite[Theorem 23]{LPPS},  weak duality solutions to \eqref{DP-ell} are such that
	\begin{align}\label{reg}
		\fl{\frac{s}{2}}{u}\in L^p(\Omega), \;\;\;\forall \;p\in(1, N/(N-s))
	\end{align}
an this implies that $u\in W^{s,p}(\Omega)$ too.

\begin{proof}[Proof of Theorem \ref{reg-p-ell}]
For the sake of completeness we include the proof. 

We start by noticing that, assuming $f\in L^p(\Omega)$, $1<p<\infty$, we have that \eqref{DP-ell} has a unique weak solution $u$ (either the finite-energy or the duality one) and that, from the discussion above, we have $u\in W^{s,p}(\Omega)$. In particular, $u\in L^p(\Omega)$. 

As we have mentioned above, our strategy is based on a cut-off argument that will allow us to show that the solutions of the fractional Dirichlet problem in $\Omega$, after cut-off, are solutions of the elliptic problem on the whole space $\RR^N$, for which Theorem \ref{re-r-N} holds. For this purpose, given $\omega$ and $\widetilde\omega$ two open subsets of the domain $\Omega$ such that $\widetilde\omega\Subset\omega\Subset\Omega$, we introduce a cut-off function $\eta\in \mathcal D(\omega)$ such that
\begin{equation}\label{eta}
\begin{cases}
\eta(x)\equiv 1\;\;\;&\mbox{ if }\; x\in\widetilde\omega\\
0\le \eta(x)\le 1&\mbox{ if }\; x\in\omega\setminus\widetilde\omega\\
\eta(x)=0&\mbox{ if }\; x\in\RR^N\setminus\omega.
\end{cases}
\end{equation}

Let $\omega$ and $\eta\in\mathcal D(\omega)$ be respectively the set and the cut-off function constructed in \eqref{eta}. We consider the function $u\eta\in W^{s,p}(\RR^N)$ and we have that $(-\Delta)^s(u\eta)$ is given by (see, e.g., \cite[Proposition 1.5]{fl_reg} or \cite{ROS})
\begin{align}\label{for-del-prod}
	\fl{s}{(u\eta)} = \eta f + u\fl{s}{\eta} - I_s(u,\eta),
\end{align}
where $I_s(u,\eta)$ is a remainder term which is given by 
\begin{align}\label{Is-formula}
I_s(u,\eta)(x):=C_{N,s}\int_{\RR^N}\frac{(u(x)-u(y))(\eta(x)-\eta(y))}{|x-y|^{N+2s}}\;dy,\;\;x\in\RR^N.
\end{align}
Let $\omega_1, \omega_2$ be open sets such that
\begin{align}\label{omega12}
\overline{\omega}\subset\omega_1\subset\overline{\omega}_1\subset\omega_2
\subset\overline{\omega}_2\subset\Omega. 
\end{align} 

Since the function $\eta$ and the set $\omega$ in \eqref{eta} are arbitrary, it follows that  $u\in W^{s,p}(\omega_2)$. Thus we have $u\in W^{s,p}(\omega_2)\cap L^p(\Omega)$. Let
\begin{align*}
g:=u(-\Delta)^s\eta  -I_s(u,\eta).
\end{align*}
We now claim that $g\in L^p(\RR^N)$ and there exists a constant $C>0$ such that
\begin{align}\label{norm-f-p}
\|g\|_{L^p(\RR^N)}\le C\left(\|u\|_{W^{s,p}(\omega_2)}+\|u\|_{L^p(\Omega)}\right).
\end{align}
Indeed, it is clear that $g$ is defined on all $\RR^N$. Moreover
\begin{align}\label{est1-p}
\|u(-\Delta)^s\eta\|_{L^p(\RR^N)}^p=\int_{\Omega}|u(-\Delta)^s\eta|^p\;dx\le \|(-\Delta)^s\eta\|_{L^\infty(\Omega)}^p\|u\|_{L^p(\Omega)}^p.
\end{align}
For estimating the term $I_s$, we use the decomposition
\begin{align*}
I_s(u,\eta)(x):=&C_{N,s}\int_{\RR^N}\frac{(u(x)-u(y))(\eta(x)-\eta(y))}{|x-y|^{N+2s}}\;dy\\
=&C_{N,s}\int_{\omega_1}\frac{(u(x)-u(y))(\eta(x)-\eta(y))}{|x-y|^{N+2s}}\;dy\\
&+C_{N,s}\eta(x)\int_{\RR^N\setminus\omega_1}\frac{u(x)-u(y)}{|x-y|^{N+2s}}\;dy =\mathbb I_1(x)+\mathbb I_2(x),\;\;x\in\RR^N,
\end{align*}
where we have set
\begin{align*}
\mathbb I_1(x):=C_{N,s}\int_{\omega_1}\frac{(u(x)-u(y))(\eta(x)-\eta(y))}{|x-y|^{N+2s}}\;dy,\;\;x\in\RR^N,
\end{align*}
and
\begin{align*}
\mathbb I_2(x):=C_{N,s}\eta(x)\int_{\RR^N\setminus\omega_1}\frac{u(x)-u(y)}{|x-y|^{N+2s}}\;dy,\;\;x\in\RR^N.
\end{align*}
Let $p':=p/(p-1)$. Using the H\"older inequality, we get that for a.e. $x\in\RR^N$,
\begin{align}\label{CS1-p}
|\mathbb I_1(x)|\le C_{N,s}\left(\int_{\omega_1}\frac{|u(x)-u(y)|^p}{|x-y|^{N+sp}}\;dy\right)^{\frac 1p}\left(\int_{\omega_1}\frac{|\eta(x)-\eta(y)|^{p'}}{|x-y|^{N+sp'}}\;dy\right)^{\frac 1{p'}}.
\end{align}

Let $x\in\omega_1$ be fixed and $R>0$ such that $\omega_1\subset B(x,R)$. Using the Lipschitz continuity of the function $\eta$, we obtain that there exists constant $C>0$ such that
\begin{align}\label{CS2-p}
\int_{\omega_1}\frac{|\eta(x)-\eta(y)|^{p'}}{|x-y|^{N+sp'}}\;dy\le C\int_{\omega_1}\frac{dy}{|x-y|^{N+sp'-p'}}\le C\int_{B(x,R)}\frac{dy}{|x-y|^{N+sp'-p'}}\le C.
\end{align}

In what follows, we will employ the following estimate. Let $A\subset\RR^N$ be a bounded set and $B\subset\RR^N$ an arbitrary set. Then there exists a constant $C>0$ (depending on $A$ and $B$) such that
\begin{align}\label{ine-dist}
|x-y|\ge C(1+|y|),\;\;\forall\;x\in A,\;\forall\;y\in \RR^N\setminus B,\;\mbox{dist}(A,\RR^N\setminus B)=\delta>0.
\end{align}
Now, using \eqref{CS1-p}, \eqref{CS2-p} and \eqref{ine-dist}, we get 
\begin{align}\label{I1-p}
\int_{\RR^N}|\mathbb I_1(x)|^p\;dx\le& C\left(\int_{\omega_2} \int_{\omega_1}\frac{|u(x)-u(y)|^p}{|x-y|^{N+sp}}\;dydx+\int_{\RR^N\setminus\omega_2} \int_{\omega_1}\frac{|u(x)-u(y)|^p}{|x-y|^{N+sp}}\;dydx\right)\notag\\
\le &C\left(\|u\|_{W^{s,p}(\omega_2)}^p+\int_{\RR^N\setminus\omega_2} \int_{\omega_1}\frac{|u(x)|^p+|u(y)|^p}{(1+|x|)^{N+sp}}\;dydx\right)\notag\\
\le &C\left(\|u\|_{W^{s,p}(\omega_2)}^p+\|u\|_{L^p(\Omega)}^p\right) ,
\end{align}
where we have also used that $u=0$ on $\RR^N\setminus\Omega$.
Recall that $\mathbb I_2=0$ on $\RR^N\setminus\omega$. 
Then using the H\"older inequality, we get that 
\begin{align}\label{E1-1}
|\mathbb I_2(x)|^p\le C\left(\int_{\RR^N\setminus\omega_1}\frac{\eta^{p'}(x)dy}{|x-y|^{N+sp'}}\right)^{p-1}\int_{\RR^N\setminus\omega_1}\frac{|u(x)-u(y)|^p}{|x-y|^{N+sp}}\;dy.
\end{align}
For any $y\in \RR^N\setminus\omega_1$, we have that
\begin{align*}
\frac{\eta^{p'}(x)}{|x-y|^{N+sp'}}=\frac{\chi_{\overline{\omega}}(x)\eta^{p'}(x)}{|x-y|^{N+sp'}}\le \chi_{\overline{\omega}}(x)\eta^{p'}(x)\sup_{x\in\overline{\omega}}\frac{1}{|x-y|^{N+sp'}}.
\end{align*}
So there exists a constant $C>0$ such that
\begin{align}\label{E2-2}
\int_{\RR^N\setminus\omega_1}\frac{\eta^{p'}(x)dy}{|x-y|^{N+sp'}}\le \chi_{\overline{\omega}}(x)\eta^{p'}(x)\int_{\RR^N\setminus\omega_1}\frac{dy}{\mbox{dist}(y,\partial\overline{\omega})^{N+sp'}}\le C\chi_{\overline{\omega}}(x)\eta^{p'}(x).
\end{align}

In \eqref{E2-2} we have also used that the integral is finite which follows from the fact that $\mbox{dist}(\partial\omega_1,\partial\overline{\omega})\ge\delta>0$ together with the fact that $\mbox{dist}(y,\partial\overline{\omega})$ grows linearly as $y$ tends to infinity and $N+sp'>N$.

Since $\chi_{\overline{\omega}}\eta^{p'}\in L^\infty(\omega)$, and using \eqref{E1-1}, \eqref{E2-2} and \eqref{ine-dist}, we also get that there exists a constant $C>0$ such that
\begin{align}\label{I2-p}
	\int_{\RR^N}|\mathbb I_2(x)|^p\;dx=&\int_{\omega}|\mathbb I_2(x)|^p\;dx\le C\int_{\omega}\int_{\RR^N\setminus\omega_1}\frac{|u(x)-u(y)|^p}{|x-y|^{N+sp}}\;dydx\notag
	\\
	\le &C\int_{\omega}\int_{\RR^N\setminus\omega_1}\frac{|u(x)|^p+|u(y)|^p}{(1+|y|)^{N+sp}}\;dydx\le C\|u\|_{L^p(\Omega)}^p,
\end{align}
where we have used again that $u=0$ on $\RR^N\setminus\Omega$.
Estimate \eqref{norm-f-p} follows from \eqref{est1-p}, \eqref{I1-p}, \eqref{I2-p} and we have shown the claim. We therefore proved that $\eta u$ is a weak solution to the Poisson equation \eqref{PE} with $F$ given by $F=\eta f+g$. Since $F\in L^p(\RR^N)$, it follows from Theorem \ref{re-r-N} that $\eta u\in \mathscr{L}^p_{2s}(\RR^N)$. We have shown that $u\in \mathscr{L}^p_{2s, \textrm{loc}}(\Omega)$. As a consequence we have the following results. 
\begin{enumerate}
\item If $1<p<2$ and $s\ne 1/2$, then $\eta u\in B^{2s}_{p,2}(\RR^N)$, hence $u\in B^{2s}_{p,2, \textrm{loc}}(\Omega)$.
\item If $1<p<2$ and $s= 1/2$, then $\eta u\in W^{2s,p}(\RR^N)=W^{1,p}(\RR^N)$, hence $u\in W^{2s,p}_{\textrm{loc}}(\Omega)=W^{1,p}_{\textrm{loc}}(\Omega)$.
\item If $2\leq p<\infty$, then $\eta u\in W^{2s,p}(\RR^N)$, hence $u\in W^{2s,p}_{\textrm{loc}}(\Omega)$.
\end{enumerate}
\noindent The proof is finished.
\end{proof}
We conclude this section mentioning that Theorem \ref{reg-p-ell} can be proved also using techniques from pseudo-differential calculus (see, e.g., \cite[Section 7]{grubb} or \cite[Chapter XI, Theorem 2.5]{taylor}).
Our approach is different and provides a proof based on basic estimates of solutions of general elliptic operators. In particular, our proofs do not require any knowledge of pseudo-differential operators theory.

\section{Proof of Theorem \ref{reg-thm-2}}\label{L2-reg}

\noindent The proof of Theorem \ref{reg-thm-2} employs a cut-off argument, as in Theorem \ref{reg-p-ell}. In particular: 
\begin{itemize}
	\item Firstly, we  treat the case $\Omega=\RR^N$, adapting the proof  in \cite[Section 7.1.3, Theorem 5]{evans} for the classical Laplace operator.
	\item The case of a general $\Omega$ is reduced to the previous one applying a cut-off argument.
\end{itemize}

\subsection{The $W^{2s,2}$-regularity on $\RR^N$}
In this Section, we prove the $W^{2s,2}$-regularity result in the case where $\Omega$ is the whole space $\RR^N$. We will adapt the proof presented in \cite[Section 7.1.3, Theorem 5]{evans} for the local case.

\begin{theorem}\label{reg-thm-r}
Assume $f\in L^2(\RR^N\times (0, T))$ and let $u\in L^2((0,T);W^{s,2}(\RR^N))\cap C([0,T];L^2(\RR^N))$ with $u_t\in L^2((0,T);W^{-s,2}(\RR^N))$ be the unique finite energy solution of the system
\begin{align}\label{DPR}
	\begin{cases}
	u_t+\fl{s}{u}=f\;\;&\mbox{ in }\RR^N\times (0,T),
		\\
		u(\cdot,0)\equiv 0\;&\mbox{ on } \RR^N.
	\end{cases}
\end{align}
Then 
\begin{align*}
u\in L^2((0,T);W^{2s,2}(\RR^N))\cap L^{\infty}((0,T);W^{s,2}(\RR^N)),\;\; u_t\in L^2(\RR^N\times (0, T))
\end{align*}
\end{theorem}

\begin{proof}
First of all, we notice that the function $v:=ue^{-t}$ solves the system
\begin{align}\label{DPRV}
	\begin{cases}
		v_t+\fl{s}{v}+v=g & \mbox{ in } \RR^N\times [0,T],
		\\
		v(\cdot,0)\equiv 0 & \mbox{ on } \RR^N,
	\end{cases}
\end{align}
with $g:=fe^{-t}\in L^2(\RR^N\times (0, T))$. Now, multiplying \eqref{DPRV} by $v_t$ and integrating by parts over $\RR^N$ we obtain that
\begin{align*}
	(v_t,v_t) + B[v,v_t] + (v,v_t) = (g,v_t),
\end{align*}
where $(\cdot,\cdot)$ is the classical scalar product on $L^2(\RR^N)$, while with $B[\cdot,\cdot]$ we indicated the bilinear form
\begin{align*}
	B[\phi,\psi]:=\frac{C_{N,s}}{2}\int_{\RR^N}\int_{\RR^N}\frac{(\phi(x)-\phi(y))(\psi(x)-\psi(y))}{|x-y|^{N+2s}}\,dxdy.
\end{align*}
Moreover, we observe that 
\begin{align*}
	B[v,v_t] = \frac{1}{2}\frac{d}{dt}B[v,v]\;\mbox{ and }\; (v,v_t)=\frac{1}{2}\frac{d}{dt}(v,v).
\end{align*}
Hence, using Young's inequality we have that, for every $\varepsilon>0$,
\begin{align*}
	\norm{v_t}{L^2(\RR^N)}^2 + \frac{1}{2}\frac{d}{dt}\Big(B[v,v] + (v,v)\Big) = (g,v_t)\leq \frac{C}{\varepsilon}\norm{g}{L^2(\RR^N)}^2 + \varepsilon\norm{v_t}{L^2(\RR^N)}^2.
\end{align*}
Choosing $\varepsilon\le 1$ and integrating in time we find that
\begin{align*}
	\int_0^T\norm{v_t}{L^2(\RR^N)}^2\,dt + \sup_{t\in[0,T]}\Big(B[v(t),v(t)] + (v(t),v(t))\Big) \leq C\int_0^T\norm{g}{L^2(\RR^N)}^2\,dt,
\end{align*}
which implies that
\begin{align*}
	\norm{v_t}{L^2(\RR^N \times (0, T))}^2 + \norm{v}{L^{\infty}((0,T),W^{s,2}(\RR^N))} \leq C\norm{g}{L^2(\RR^N \times (0, T))}^2.
\end{align*}
Therefore, 
\begin{align*}
	v\in L^{\infty}((0,T);W^{s,2}(\RR^N)),\;\;\; v_t\in L^2(\RR^N\times (0, T))
\end{align*}
and, by definition,  $u$ has the same regularity too. Finally, the $W^{2s,2}$ regularity for $u$ in the space variable is obtained in the following way. From \eqref{DPR} we have that $\fl{s}{u} = f-u_t\in L^2(\RR^N\times (0, T))$. Hence, a. e. $t\in(0,T)$, we have that $\fl{s}{u}(\cdot,t) = h(\cdot,t)\in L^2(\RR^N)$ and, applying the regularity results for the elliptic case (see Theorem \ref{re-r-N}) we get that $u(\cdot,t)\in W^{2s,2}(\RR^N)$ a. e. $t\in (0,T)$. Furthermore $u\in L^2((0,T);W^{2s,2}(\RR^N))$ and the proof is finished.
\end{proof}

\subsection{The $W^{2s,2}_{\textrm{loc}}$-regularity in $\Omega$}

\begin{proof}[\bf Proof of Theorem \ref{reg-thm-2}]

As we have mentioned above, our strategy is based on a cut-off argument that will allow us to show that solutions of the fractional parabolic problem in $\Omega$, after cut-off, are solutions of a problem on the whole space $\RR^N$, for which Theorem \ref{reg-thm-r} holds. 

Let $f\in L^2(\Omega\times (0,T))$ and $u\in L^2((0,T);W_0^{s,2}(\bOm))\cap C([0,T];L^2(\Omega))$ with $u_t\in L^2((0,T);W^{-s,2}(\bOm))$ be the unique finite energy solution to the system \eqref{DP}.  

Let $\omega$ and $\eta\in\mathcal D(\omega)$ be respectively the set and the cut-off function constructed in \eqref{eta}. We consider the function $v:=u\eta$ and we write the equation satisfied by $v$. Recall from \eqref{for-del-prod} that the fractional Laplacian of $v$ is given by 
\begin{align*}
\fl{s}{v}=\fl{s}{(u\eta)}=u\fl{s}{\eta}+\eta\fl{s}{u}-I_s(u,\eta),
\end{align*}
where the remainder term $I_s$ has been defined in \eqref{Is-formula}. Then, $v$ is a solution to the following problem on $\RR^N$:
\begin{align}
	\begin{cases}
		v_t+\fl{s}{v}=F\ & \mbox{ in }\RR^N\times (0,T),
		\\
		v(\cdot,0)\equiv 0 & \mbox{ on } \RR^N,
	\end{cases}
\end{align}
with $F=\eta f + u(-\Delta)^s\eta - I_s(u,\eta)$. 

Following the proof of \cite[Theorem 1.2]{fl_reg}, we can show that $F\in L^2(\RR^N\times (0,T))$. Hence, from Theorem \ref{reg-thm-r} we obtain that 
\begin{align*}
	v\in L^2((0,T);W^{2s,2}(\RR^N))\cap L^{\infty}((0,T);W^{s,2}(\RR^N)),\;\;\ v_t\in L^2(\RR^N\times (0,T)).
\end{align*}

This implies that $u\in L^2((0,T);W^{2s,2}_{\textrm{loc}}(\Omega))\cap L^{\infty}((0,T);W_0^{s,2}(\bOm))$ and $u_t\in L^2(\Omega\times (0,T))$. The proof is finished.
\end{proof}

\section{Proof of Theorem \ref{reg-thm-p}}\label{Lp-reg}

In this section, we prove the local regularity for the solutions to the parabolic problem \eqref{DP}, corresponding to a right hand side $f\in L^p(\Omega\times (0,T))$, with $1<p<\infty$. 

First of all, notice that the following developments  also apply to the case $p=2$. This special case has already been treated in the previous section, and there the proof of our local regularity Theorem \ref{reg-thm-2} has been developed taking advantage of the Hilbert structure of the spaces $L^2(\Omega)$ and $L^2(\Omega\times (0,T))$. 

Clearly that strategy cannot be extended to the general $L^p$ setting, and we have to adopt a different approach. This approach relies on an abstract result due to Lamberton \cite{lamberton}. In particular, the proof of Theorem \ref{reg-thm-p} will be a direct consequence of \cite[Theorem 1]{lamberton}. For the sake of completeness, we recall its statement here.

\begin{theorem}\label{thm-lamberton}
Let $(\Omega,\Sigma,m)$ be a measure space and let $A$ be the generator of a strongly continuous semigroup of linear operators $(\mathbb{T}_t)_{t\geq 0}$ on $L^2(\Omega,\Sigma,m)$ satisfying the following hypothesis:
\begin{enumerate}
	\item The semigroup $(\mathbb{T}_t)_{t\geq 0}$ is analytic and bounded on $L^2(\Omega,\Sigma,m)$.
	\item For every $p\in [1,\infty]$ and  $\phi\in L^p(\Omega)\cap L^2(\Omega)$ we have the estimate
\begin{align*}
\norm{\mathbb{T}_t\phi}{L^p(\Omega)}\leq \norm{\phi}{L^p(\Omega)}, \mbox{ for all }\; t\geq 0.
\end{align*}
\end{enumerate} 
Let $p\in(1,\infty)$. If $f\in L^p(\Omega\times (0,T))$, then the system
\begin{align*}
	\begin{cases}
		u_t-Au=f,& t\in (0,T)
		\\
		u(0)=0
	\end{cases}
\end{align*}
admits a solution $u\in C([0,T];L^p(\Omega))$, such that $u_t$, $Au\in L^p(\Omega\times (0,T))$.
\end{theorem} 

\begin{proof}[\bf Proof of Theorem \ref{reg-thm-p}]
First of all notice that the operator $A=-\fl{s}$ with domain
\begin{align}\label{dom-op}
	\mathcal{D}(A) = \Big\{u\in W_0^{s,2}(\bOm),\,\;\;\fl{s}{u}\in L^2(\Omega)\Big\} 
\end{align}
is the generator of a submarkovian strongly continuous semigroup $(\mathbb{T}_t)_{t\geq 0}$ which is also ultracontractive (see, e.g., \cite[Lemma 2.4]{fl_reg}). Let $f\in L^p(\Omega\times (0,T))$ and let $u$ be the corresponding weak solution to the system \eqref{DP}. Then, it follows from Theorem \ref{thm-lamberton} that $u_t,\fl{s}{u}\in L^p(\Omega\times (0,T))$. In particular we have that $\fl{s}{u}(\cdot,t)=(f-u_t)(\cdot,t)\in L^p(\Omega)$ a. e. $t\in (0,T)$ and, according to Theorem \ref{reg-p-ell}, this implies that
$u(\cdot,t)\in \mathscr{L}^p_{2s, {\textrm{loc}}}(\Omega)$  a. e. $t\in(0,T)$. Therefore, for all $t\in(0,T)$ we have the following results.
\begin{enumerate}
\item[(i)] If $1<p<2$ and $s\ne 1/2$, then $u(\cdot,t)\in B^{2s}_{p,2, \textrm{loc}}(\Omega)$, a. e. $t\in (0,T)$.
\item[(ii)] If $1<p<2$ and $s= 1/2$, then $u(\cdot,t)\in W^{2s,p}_{\textrm{loc}}(\Omega)=W^{1,p}_{\textrm{loc}}(\Omega)$, a. e. $t\in (0,T)$.
\item[(iii)] If $2\leq p<\infty$, then $u(\cdot,t)\in W^{2s,p}_{\textrm{loc}}(\Omega)$, a. e. $t\in (0,T)$.
\end{enumerate}
Furthermore:
\begin{enumerate}
\item If $1<p<2$ and $s\ne 1/2$, then $u\in L^p((0,T);B^{2s}_{p,2, \textrm{loc}}(\Omega))$.
\item If $1<p<2$ and $s= 1/2$, then $u\in L^p((0,T);W^{2s,p}_{\textrm{loc}}(\Omega))=L^p((0,T);W^{1,p}_{\textrm{loc}}(\Omega))$.
\item If $2\leq p<\infty$, then $u\in L^p((0,T);W^{2s,p}_{\textrm{loc}}(\Omega))$.
\end{enumerate}
The proof of the theorem is finished.
\end{proof}

\noindent We conclude this section with the following remark.

\begin{remark}
{\em Recall that we have said in the proof of Theorem \ref{reg-thm-p} that the operator $A=-\fl{s}$ with domain given by \eqref{dom-op} generates a strongly continuous submarkovian semigroup $(\mathbb{T}_t)_{t\geq 0}$ on $L^2(\Om)$ and the semigroup is analytic and  ultracontractive. This implies that the semigroup can be extended to contraction semigroups on $L^p(\Omega)$ for all $p\in [1,\infty]$ and each semigroup is strongly continuous if $p\in [1,\infty)$ and bounded analytic if $p\in (1,\infty)$. Let $A_p$ denote the generator of the semigroup on $L^p(\Omega)$ for $p\in [1,\infty]$ so that $A_2$ coincides with $A$. By Theorem \ref{thm-lamberton} if $1<p<\infty$  and $f\in L^p(\Omega\times (0,T))$, then the unique solution $u\in C([0,T];L^p(\Omega))$ of the system \eqref{DP} has the following regularity:
\begin{align*}
u\in L^p((0,T); D(A_p)).
\end{align*}
This trivially implies that $A_pu\in L^p(\Omega\times (0,T))$ and $u_t\in L^p(\Omega\times (0,T))$. Our contribution in the present paper was to show that  $D(A_p)\subset \mathscr{L}^{p}_{2s, {\rm loc}}(\Omega)$ for every $1<p<\infty$.
}
\end{remark}

\section{Open problems and perspectives}\label{open_pb}

In the present paper we proved that weak solutions to the parabolic problem for the fractional Laplacian, with a non-homogeneous right-hand side $f\in L^p(\Omega\times (0,T))$ ($1<p<\infty$) and zero initial datum, belong to $L^p((0,T);\mathscr{L}^p_{2s, {\textrm{loc}}}(\Omega))$.
The following comments are worth considering.

\begin{enumerate}
	\item A natural interesting extension of our result would be the analysis of the global maximal regularity in space for weak solutions to \eqref{DP}. The problem is delicate however. 

Indeed, already at the elliptic level, we know that even if $\Omega$ has a smooth boundary, then the global maximal regularity up to the boundary does not hold. To be more precise, assume that $\Omega$ has a smooth boundary, $f\in L^p(\Omega)$ ($1<p<\infty$) and let $u$ be the associated weak solution to the Dirichlet problem \eqref{DP-ell}. It is known that, if $p\ge 2$, then $u$ does not always belongs to $W^{2s,p}(\Omega)$ and, if $1<p<2$, then $u$ does not always belong to $B_{p,2}^{2s}(\Omega)$.
This shows that in general, the corresponding weak solution $u$ to the parabolic system \eqref{DP} does not always belong to $L^p((0,T);W^{2s,p}(\Omega))$ if $p\ge 2$ and does not always  belong to  $L^p((0,T);B_{p,2}^{2s}(\Omega))$ if $1<p< 2$. 
	
On the one hand, Theorem \ref{thm-lamberton} shows that $u\in L^p((0,T);D(A_p))$, that is, in particular $u(\cdot,t)\in D(A_p)$ for a.e. $t\in (0,T)$. On the other hand, according to the discussions given in \cite[Section 5]{fl_reg}, at least if $\Omega$ has a sufficiently smooth boundary, one has that $u(\cdot,t)=\rho^sv(\cdot,t)$ where $v(\cdot,t)$ is a regular function up to the boundary. Here, $\rho(x):=\textrm{dist}(x,\partial\Omega)$ for $x\in\Omega$. In addition one could expect that $\rho^{-s}u, \rho^{1-s}u\in L^p((0,T);L^p((0,T);W^{2s,p}(\Omega))$ if $2\le p<\infty$, and $\rho^{-s}u, \rho^{1-s}u\in L^p((0,T);L^p((0,T);B_{p,2}^{2s}(\Omega))$ if $1< p<2$. This constitutes an interesting open problem.  We refer to \cite[Section 5]{fl_reg} for a more complete discussion on related topics and the difficulties that it raises.
	
	\item It would be interesting to consider the case of a non-zero initial datum in equation \eqref{DP}. In the Hilbert space framework, i.e. when working in the $L^2(\Omega)$ setting, the strategy of Section \ref{L2-reg} can be extended to deal with initial data in $W_0^{s,2}(\bOm)$. To the best of our knowledge, the corresponding analogous result in the case $p\in (1,\infty)$, $p\ne 2$, is still unknown.
\end{enumerate}

\end{document}